\documentclass[oneside,english]{amsart}
\usepackage[T1]{fontenc}
\usepackage[latin9]{inputenc}
\usepackage{geometry}
\usepackage{amstext}
\usepackage{amsthm}
\usepackage{amssymb}

\makeatletter
\numberwithin{equation}{section}
\numberwithin{figure}{section}
\theoremstyle{plain}
\newtheorem{thm}{\protect\theoremname}
  \theoremstyle{remark}
  \newtheorem*{acknowledgement*}{\protect\acknowledgementname}
  \theoremstyle{plain}
  \newtheorem{lem}[thm]{\protect\lemmaname}
  \theoremstyle{plain}
  \newtheorem{prop}[thm]{\protect\propositionname}
  \theoremstyle{remark}
  \newtheorem{rem}[thm]{\protect\remarkname}
  \theoremstyle{remark}
  \newtheorem{notation}[thm]{\protect\notationname}
  \theoremstyle{plain}
  \newtheorem{cor}[thm]{\protect\corollaryname}
  \theoremstyle{definition}
  \newtheorem{example}[thm]{\protect\examplename}

\makeatother

\usepackage{babel}
  \providecommand{\acknowledgementname}{Acknowledgement}
  \providecommand{\corollaryname}{Corollary}
  \providecommand{\examplename}{Example}
  \providecommand{\lemmaname}{Lemma}
  \providecommand{\notationname}{Notation}
  \providecommand{\propositionname}{Proposition}
  \providecommand{\remarkname}{Remark}
\providecommand{\theoremname}{Theorem}

\begin{document}


\thanks{This work received support from the French \textquotedbl{}Investissements
d\textquoteright Avenir\textquotedbl{} program, project ISITE-BFC
(contract ANR-lS-IDEX-OOOB)}

\author{Adrien Dubouloz}

\address{IMB UMR5584, CNRS, Univ. Bourgogne Franche-Comté, F-21000 Dijon,
France.}

\email{adrien.dubouloz@u-bourgogne.fr}

\author{Gene Freudenburg }

\address{Department of Mathematics Western Michigan University Kalamazoo,
Michigan 49008. }

\email{gene.freudenburg@wmich.edu }

\author{Lucy Moser-Jauslin}

\address{IMB UMR5584, CNRS, Univ. Bourgogne Franche-Comté, F-21000 Dijon,
France.}

\email{lucy.moser-jauslin@u-bourgogne.fr}

\dedicatory{In memory of Mariusz Koras }

\title[smooth rational varieties with infinitely many real forms ]{Algebraic vector bundles on the $2$-sphere and smooth rational
varieties with infinitely many real forms }
\begin{abstract}
We construct smooth rational real algebraic varieties of every dimension
$\geq4$ which admit infinitely many pairwise non-isomorphic real
forms. 
\end{abstract}

\maketitle

\section*{Introduction}

A classical problem in real algebraic geometry is the classification
of real forms of a given real algebraic variety $X$, that is, real
algebraic varieties $Y$ non isomorphic to $X$ but whose complexifications
$Y_{\mathbb{C}}$ are isomorphic to $X_{\mathbb{C}}$ as complex algebraic
varieties. For example, the smooth real affine algebraic surfaces
$\mathbb{S}^{2}=\left\{ x^{2}+y^{2}+z^{2}=1\right\} $ and $\mathcal{D}=\left\{ uv+z^{2}=1\right\} $
in $\mathbb{A}_{\mathbb{R}}^{3}$ have isomorphic complexifications,
an explicit isomorphism being simply given by the linear change of
complex coordinates $u=x+iy$ and $v=x-iy$, but are non isomorphic.
This follows for instance from the fact that the set of real points
of $\mathbb{S}^{2}$ is the usual euclidean $2$-sphere $S^{2}\subset\mathbb{R}^{3}$
whereas the set of real points of $\mathcal{D}$ is not compact for
the Euclidean topology. 

Examples of smooth real projective varieties admitting infinitely
many pairwise non-isomorphic real forms were only found very recently
successively by Lesieutre \cite{Le16} in dimension $\geq6$ and by
Dinh-Oguiso \cite{DO17} in every dimension $\geq2$. These are obtained
as a by-product of clever constructions of smooth complex projective
algebraic varieties defined over $\mathbb{R}$ with discrete but non
finitely generated automorphism groups containing infinitely many
conjugacy classes algebraic involutions. All their examples are non
geometrically rational and to our knowledge, the question of existence
of rational real algebraic varieties, projective or not, with infinitely
many real forms was left open. Our first main result explicitly fills
this gap for smooth real affine fourfolds:
\begin{thm}
\label{thm:MainTh-1}The smooth rational real affine fourfold $\mathbb{S}^{2}\times\mathbb{A}_{\mathbb{R}}^{2}$
has at least countably infinitely many pairwise non-isomorphic real
forms.
\end{thm}

In contrast with the examples found by Lesieutre and Dinh-Oguiso,
which rely on constructions of special classes of complex projective
varieties by techniques of birational geometry, ours are inspired
by basic results on the classification of topological vector bundles
on the real sphere $S^{2}\subset\mathbb{R}^{3}$. Our construction
can indeed be interpreted as a sort of ``algebraization'' of the
property that the complexification $E\otimes_{\mathbb{R}}\mathbb{C}$
of any topological real vector bundle $\pi:E\rightarrow S^{2}$ of
rank $2$ on $S^{2}$ is isomorphic, as a topological real vector
bundle of rank $4$, to the trivial bundle $S^{2}\times\mathbb{R}^{4}$.
More precisely, we show that the topological real vector bundles of
rank $2$ on $S^{2}$, which are nothing but the underlying real vector
bundles of the complex line bundles $\mathcal{O}_{\mathbb{CP}^{1}}(n)$,
$n\geq0$, over $\mathbb{CP}^{1}\simeq S^{2}$, admit algebraic models
in the form of algebraic vector bundles $p_{n}:V_{n}\rightarrow\mathbb{S}^{2}$
of rank $2$ on $\mathbb{S}^{2}$ with pairwise non-isomorphic total
spaces, whose complexifications $p_{\mathbb{C}}:V_{n,\mathbb{C}}\rightarrow\mathbb{S}_{\mathbb{C}}^{2}$
are all isomorphic to the trivial bundle $\mathbb{S}_{\mathbb{C}}^{2}\times\mathbb{A}_{\mathbb{C}}^{2}$. 

It is worth noticing that by a result of Kambayashi \cite{Kam75},
$\mathbb{A}_{\mathbb{R}}^{2}$ has no nontrivial real form. One can
check along the same lines using the fact that similarly as to $\mathrm{Aut}(\mathbb{A}_{\mathbb{C}}^{2})$,
the automorphism group $\mathrm{Aut}(\mathbb{S}_{\mathbb{C}}^{2})$
of $\mathbb{S}_{\mathbb{C}}^{2}\simeq\mathcal{D}_{\mathbb{C}}$ has
a structure of a free product of two subgroups amalgamated along their
intersection \cite{BD11, ML90}, that $\mathcal{D}$ is the unique
nontrivial real form of $\mathbb{S}^{2}$. So while $\mathbb{A}_{\mathbb{R}}^{2}$
and $\mathbb{S}^{2}$ both have finitely many real forms, the total
spaces of the algebraic vector bundles $p_{n}:V_{n}\rightarrow\mathbb{S}^{2}$
provide an infinite countable family of real forms of $\mathbb{S}^{2}\times\mathbb{A}_{\mathbb{R}}^{2}$
which are by construction pairwise locally isomorphic over $\mathbb{S}^{2}$,
but globally pairwise non-isomorphic as real algebraic varieties.
In contrast, reminiscent of the fact that for every $r\geq3$ there
exists a unique nontrivial topological real vector bundle of rank
$r$ on $S^{2}$, it turns out that the varieties $V_{n}\times\mathbb{A}_{\mathbb{R}}^{r-2}$,
$n\geq0$, give rise to a unique class of nontrivial real form of
$\mathbb{S}^{2}\times\mathbb{A}_{\mathbb{R}}^{r}$ (see Corollary
\ref{cor:Stable-Iso} below). 

Our construction thus does not directly yield higher dimensional families
of examples by simply taking product with affine spaces. Nevertheless,
a suitable adaptation of the technique used by Dinh-Oguiso \cite{DO17},
consisting in our situation of taking products of the $V_{n}$ with
well-chosen real rational affine varieties of log-general type, allows
us to derive the following general existence result: 
\begin{thm}
\label{thm:MainTh-2}For every $d\geq4$, there exist smooth rational
real affine varieties of dimension $d$ which have at least countably
infinitely many pairwise non-isomorphic real forms. 
\end{thm}

The article is organized as follows. The first section contains a
short review of the classical correspondence between quasi-projective
real algebraic varieties and quasi-projective complex varieties endowed
with a real structure as well as a recollection on Euclidean topologies
of real and complex algebraic varieties. Section 2 is devoted to the
construction of algebraic models of topological real vector bundles
over the $2$-sphere $S^{2}\subset\mathbb{R}^{3}$. The existence
of such models was known after successive works of Fossum \cite{Fo69}
and Moore \cite{Mo71} and, later on, of Swan \cite{Sw93}, but we
give a new geometric construction in the framework of complex varieties
with real structures which we find more transparent. Theorem \ref{thm:MainTh-1}
is then established in Section 3. Section 4 contains the proof of
Theorem \ref{thm:MainTh-2} and a complement to Theorem \ref{thm:MainTh-1}
consisting of explicit formulas for the real structures on $\mathbb{S}_{\mathbb{C}}^{2}\times\mathbb{A}_{\mathbb{C}}^{2}$
corresponding to the real algebraic vector bundles $p_{n}:V_{n}\rightarrow\mathbb{S}^{2}$. 
\begin{acknowledgement*}
The main ideas of the present article were discussed between the authors
at the occasion of the conference ``Algebraic Geometry - Mariusz
Koras in memoriam'' held at the IMPAN, Warsaw in May 2018. We are
grateful to the organizers of the conference for giving us the opportunity
to have such discussions and to the IMPAN for its support and hospitality. 
\end{acknowledgement*}

\section{Preliminaries}

In this article, the term $\mathbf{k}$-variety will always refer
to a geometrically integral quasi-projective scheme $X$ of finite
type over a base field $\mathbf{k}$ of characteristic zero. A morphism
of $\mathbf{k}$-varieties is a morphism of $\mathbf{k}$-schemes.
In the sequel, $\mathbf{k}$ will be equal to either $\mathbb{R}$
or $\mathbb{C}$, and we will say that $X$ is a real, respectively
complex, algebraic variety. To fix the notation, we let $c:\mathrm{Spec}(\mathbb{C})\rightarrow\mathrm{Spec}(\mathbb{R})$
be the \'etale double cover induced by the inclusion $\mathbb{R}\hookrightarrow\mathbb{C}=\mathbb{R}[i]/(i^{2}+1)$
and we let $\tau:\mathrm{Spec}(\mathbb{C})\rightarrow\mathrm{Spec}(\mathbb{C})$,
$i\mapsto-i$ be the usual complex conjugation. 

\subsection{Complex varieties with real structures}

Recall \cite{BS64} and \cite[Expos\'e VIII]{SGA1} that \'etale
descent for the Galois cover $c:\mathrm{Spec}(\mathbb{C})\rightarrow\mathrm{Spec}(\mathbb{R})$
provides an equivalence between the category of quasi-projective real
algebraic varieties and the category of complex algebraic varieties
equipped with a descent datum with respect to $c$. Such a descent
datum on a quasi-projective complex algebraic variety $f:V\rightarrow\mathrm{Spec}(\mathbb{C})$
is in turn uniquely determined by an isomorphism of $\mathbb{R}$-schemes
$\sigma:V\rightarrow V$ such $f\circ\sigma=\tau\circ f$ and that
satisfies the cocycle relation $\sigma^{2}=\mathrm{id}_{V}$. In other
words, $\sigma$ is an anti-regular involution of $V$, usually referred
to as a \emph{real structure} on $V$. 

For every real algebraic variety $X$, the \emph{complexification}
$X_{\mathbb{C}}=X\times_{\mathrm{Spec}(\mathbb{R})}\mathrm{\mathrm{Spec}(\mathbb{C})}$
of $X$ is canonically endowed with a real structure $\sigma_{X}=\mathrm{id}_{X}\times\tau$.
Conversely, for every complex variety $f:V\rightarrow\mathrm{Spec}(\mathbb{C})$
endowed with a real stucture $\sigma$, the ``quotient'' $q:V\rightarrow V/\langle\sigma\rangle$
exists in the category of schemes and the structure morphism $f:V\rightarrow\mathrm{Spec}(\mathbb{C})$
descends to a morphism $\overline{f}:V/\langle\sigma\rangle\rightarrow\mathrm{Spec}(\mathbb{R})=\mathrm{Spec}(\mathbb{C})/\langle\tau\rangle$
making $V/\langle\sigma\rangle$ into a real algebraic variety $X$
such that $V\simeq X_{\mathbb{C}}$. 

Two real structures $\sigma$ and $\sigma'$ on a same complex algebraic
variety $f:V\rightarrow\mathrm{Spec}(\mathbb{C})$ are called equivalent
if the associated real algebraic varieties $V/\langle\sigma\rangle$
and $V/\langle\sigma'\rangle$ are isomorphic, which holds if and
only if there exists an automorphism of complex algebraic varieties
$h:V\rightarrow V$ such that $\sigma'\circ h=h\circ\sigma$. A \emph{real
form} of a real algebraic variety $X$ is a real algebraic variety
$X'$ such that the complex varieties $X_{\mathbb{C}}$ and $X'_{\mathbb{C}}$
are isomorphic. Galois descent then provides a one-to-one correspondence
between isomorphism classes of real forms of a given real variety
$X$ and equivalence classes of real structures on its complexification
$X_{\mathbb{C}}$. 

\subsection{\label{subsec:Galois-descent-VB}Galois descent for vector bundles }

Given a real algebraic variety $X$, \'etale descent for the Galois
cover $c:\mathrm{Spec}(\mathbb{C})\rightarrow\mathrm{Spec}(\mathbb{R})$
also provides an equivalence between the category of quasi-coherent
$\mathcal{O}_{X}$-modules and the category of pairs $(\mathcal{F},\varphi)$
consisting of a quasi-coherent $\mathcal{O}_{X_{\mathbb{C}}}$-module
$\mathcal{F}$ and an isomorphism $\varphi:\mathcal{F}\stackrel{\sim}{\rightarrow}\sigma_{X}^{*}\mathcal{F}$
of $\mathcal{O}_{X_{\mathbb{C}}}$-modules such that $\varphi^{2}=(\sigma_{X}^{*}\varphi)\circ\varphi$. 

Since $\sigma_{X}$ is an involution, for every quasi-coherent $\mathcal{O}_{X_{\mathbb{C}}}$-module
$\mathcal{E}$, we have $\sigma_{X}^{*}(\sigma_{X}^{*}\mathcal{E})\simeq\mathcal{E}$.
Letting $\mathcal{F}=\mathcal{E}\oplus\sigma_{X}^{*}\mathcal{E}$,
the isomorphism $\varphi:\mathcal{F}=\mathcal{E}\oplus\sigma_{X}^{*}\mathcal{E}\rightarrow\sigma_{X}^{*}\mathcal{F}=\sigma_{X}^{*}\mathcal{E}\oplus\mathcal{E}$
exchanging the two factors of the direct sum satisfies $\varphi^{2}=(\sigma_{X}^{*}\varphi)\circ\varphi$.
We denote the corresponding quasi-coherent $\mathcal{O}_{X}$-module
by $\mathcal{E}_{\mathbb{R}}$. 

In the sequel, we essentially use this construction in the special
case where $\mathcal{E}$ is the locally free $\mathcal{O}_{X_{\mathbb{C}}}$-module
of germs of sections of a vector bundle $p:E=\mathrm{Spec}(\mathrm{Sym}^{\cdot}\mathcal{E}^{\vee})\rightarrow X_{\mathbb{C}}$
on $X_{\mathbb{C}}$ of finite rank $r$. In this geometric context,
the isomorphism $\varphi$ can be interpreted as endowing the rank
$2r$ vector bundle $\rho=p\oplus\sigma_{X}^{*}p:E\oplus\sigma_{X}^{*}E\rightarrow X_{\mathbb{C}}$
with a lift of $\sigma_{X}$ to a real structure $\tilde{\sigma}:E\oplus\sigma_{X}^{*}E\rightarrow E\oplus\sigma_{X}^{*}E$
which is linear on the fibers of $\rho$, in such a way that $\rho$
descends to a vector bundle 
\[
p_{\mathbb{R}}:E_{\mathbb{R}}:=\mathrm{Spec}(\mathrm{Sym}^{\cdot}\mathcal{E}_{\mathbb{R}}^{\vee})\simeq(E\oplus\sigma_{X}^{*}E)/\langle\tilde{\sigma}\rangle\rightarrow X\simeq X_{\mathbb{C}}/\langle\sigma_{X}\rangle
\]
of rank $2r$ on $X$. 

\subsection{Euclidean topologies}

Recall that the set $X(\mathbb{R})$ of real points of a real algebraic
variety $X$ is endowed in a natural way with the Euclidean topology,
locally induced on each affine open subset by the usual Euclidean
topology on the set $\mathbb{A}_{\mathbb{R}}^{n}(\mathbb{R})\simeq\mathbb{R}^{n}$.
The so-constructed topology on $X(\mathbb{R})$ is well-defined and
independent of the choices made \cite[Lemme 1 and Proposition 2]{Se55}.
Similarly, the set of complex points $V(\mathbb{C})$ of a complex
algebraic variety $V$ is endowed with the Euclidean topology locally
induced by that on $\mathbb{A}_{\mathbb{C}}^{n}(\mathbb{C})\simeq\mathbb{C}^{n}\simeq\mathbb{R}^{2n}$.
If $X$ (resp. $V$) is smooth, then $X(\mathbb{R})$ (resp. $V(\mathbb{C})$)
can be further equipped with a natural structure of smooth manifold
locally inherited from that on $\mathbb{R}^{n}$ (resp. $\mathbb{R}^{2n}$).
Every morphism $h\colon X\rightarrow Y$ of smooth real algebraic
varieties induces a continuous map $h(\mathbb{R})\colon X(\mathbb{R})\rightarrow Y(\mathbb{R})$
for the Euclidean topologies, which is a diffeomorphism when $f$
is an isomorphism. Similarly, a morphism of complex varieties $h:V\rightarrow W$
induces a continuous map $h(\mathbb{C}):V(\mathbb{C})\rightarrow W(\mathbb{C})$
which is a diffeomorphism when $f$ is an isomorphism. 

If $V$ is a smooth complex variety equipped with a real structure
$\sigma$, then $\sigma$ induces a smooth involution of $V(\mathbb{C})$.
The set $V(\mathbb{C})^{\sigma}$ of fixed points of $\sigma$ is
called the \emph{real locus} of $(V,\sigma)$. The quotient map $q:V\rightarrow X=V/\langle\sigma\rangle$
restricts to a diffeomorphism between $V(\mathbb{C})^{\sigma}$ endowed
with the induced smooth structure and the set of real points $X(\mathbb{R})$
of $X$ endowed with its smooth structure. \\

Let $X$ be a smooth real algebraic variety, let $p:E\rightarrow X_{\mathbb{C}}$
be a vector bundle of rank $r$ on $X_{\mathbb{C}}$ and let $p_{\mathbb{R}}:E_{\mathbb{R}}\rightarrow X$
be the vector bundle of rank $2r$ on $X$ descended from the the
rank $2r$ vector bundle $p\oplus\sigma_{X}^{*}p:E\oplus\sigma^{*}E\rightarrow X_{\mathbb{C}}$
as in \S \ref{subsec:Galois-descent-VB}. Then $p_{\mathbb{R}}(\mathbb{R}):E_{\mathbb{R}}(\mathbb{R})\rightarrow X(\mathbb{R})$
is a topological real vector bundle of rank $2r$ on the smooth manifold
$X(\mathbb{R})$. On the other hand, the restriction of $p(\mathbb{C}):E(\mathbb{C})\rightarrow X_{\mathbb{C}}(\mathbb{C})$
to the real locus $X_{\mathbb{C}}(\mathbb{C})^{\sigma_{X}}$ of $X_{\mathbb{C}}(\mathbb{C})$
defines through the diffeomorphism $X_{\mathbb{C}}(\mathbb{C})^{\sigma_{X}}\stackrel{\sim}{\rightarrow}X(\mathbb{R})$
a topological complex vector bundle $\tilde{p}:\tilde{E}\rightarrow X(\mathbb{R})$
of rank $r$ on $X(\mathbb{R})$, hence, forgetting about the complex
structure, a topological real vector bundle of rank $2r$. 
\begin{lem}
\label{lem:Top-Down-Bundle} With the notation above, $\tilde{p}:\tilde{E}\rightarrow X(\mathbb{R})$
and $p_{\mathbb{R}}(\mathbb{R}):E_{\mathbb{R}}(\mathbb{R})\rightarrow X(\mathbb{R})$
are isomorphic topological real vector bundles of rank $2r$ on $X(\mathbb{R})$. 
\end{lem}

\begin{proof}
Let $\tilde{\sigma}$ be the lift of $\sigma_{X}$ to a real structure
$\tilde{\sigma}:E\oplus\sigma_{X}^{*}E\rightarrow E\oplus\sigma_{X}^{*}E$
as in \S \ref{subsec:Galois-descent-VB}. Since $\sigma_{X}$ acts
trivially on the real locus $X_{\mathbb{C}}(\mathbb{C})^{\sigma_{X}}\simeq X(\mathbb{R})$
of $X_{\mathbb{C}}$, the restriction of $E(\mathbb{C})\oplus\sigma^{*}E(\mathbb{C})$
to $X_{\mathbb{C}}(\mathbb{C})^{\sigma_{X}}$ is equal to $\tilde{E}\oplus\tilde{E}$
on which the restriction of $\tilde{\sigma}$ acts by the involution
$j$ exchanging the two factors. By construction $p_{\mathbb{R}}(\mathbb{R}):E_{\mathbb{R}}(\mathbb{R})\rightarrow X(\mathbb{R})$
is isomorphic to the quotient bundle $(\tilde{E}\oplus\tilde{E})/\langle j\rangle\rightarrow X(\mathbb{R})$,
and the composition of the diagonal embedding $\tilde{E}\rightarrow\tilde{E}\oplus\tilde{E}$
with the quotient morphism $\tilde{E}\oplus\tilde{E}\rightarrow(\tilde{E}\oplus\tilde{E})/\langle j\rangle$
induces an isomorphism of topological real vector bundle between $\tilde{p}:\tilde{E}\rightarrow X(\mathbb{R})$
and $p_{\mathbb{R}}(\mathbb{R}):E_{\mathbb{R}}(\mathbb{R})\rightarrow X(\mathbb{R})$. 
\end{proof}

\section{Algebraic models of topological vector bundles on the $2$-sphere}

The real $2$-sphere $S^{2}=\left\{ (x,y,z)\in\mathbb{R}^{3},\,x^{2}+y^{2}+z^{2}=1\right\} $
equipped with its usual structure of smooth manifold induced by the
standard smooth structure on $\mathbb{R}^{3}$ is diffeomorphic to
set of real points $\mathbb{Q}^{2}(\mathbb{R})$ of the smooth projective
quadric surface $\mathbb{Q}^{2}\subset\mathbb{P}_{\mathbb{R}}^{3}=\mathrm{Proj}_{\mathbb{R}}(\mathbb{R}[X,Y,Z,T])$
defined by the equation $X^{2}+Y^{2}+Z^{2}-T^{2}=0$, endowed with
its Euclidean topology. The complement of the hyperplane section $H=\{T=0\}$
of $\mathbb{Q}^{2}$ is isomorphic to the smooth real affine quadric
surface $\mathbb{S}^{2}=\mathrm{Spec}(\mathbb{R}[x,y,z]/(x^{2}+y^{2}+z^{2}-1))$.
The divisor class group of $\mathbb{Q}^{2}$ is isomorphic to $\mathbb{Z}$,
generated by the class of $H$, from which it follows that the divisor
class group of $\mathbb{S}^{2}$ is trivial. Furthermore, since $H$
is a conic without real point, the inclusion $\mathbb{S}^{2}\hookrightarrow\mathbb{Q}^{2}$
induces a diffeomorphism $\mathbb{S}^{2}(\mathbb{R})\stackrel{\simeq}{\rightarrow}\mathbb{Q}^{2}(\mathbb{R})\simeq S^{2}$.

Every real algebraic vector bundle $F\rightarrow\mathbb{S}^{2}$ gives
rise to a topological real vector bundle of the same rank $F(\mathbb{R})\rightarrow\mathbb{S}^{2}(\mathbb{R})$
on $\mathbb{S}^{2}(\mathbb{R})\simeq S^{2}$. It was shown by Moore
\cite{Mo71} (see also Fossum \cite{Fo69}) that every topological
real vector bundle $\pi:E\rightarrow S^{2}$ on $S^{2}$ is isomorphic
to one obtained in this way. In other words, every topological real
vector bundle $\pi:E\rightarrow S^{2}$ admits an algebraic model
in the form of an algebraic vector bundle on $\mathbb{S}^{2}$. Later
on, Barge and Ojanguren \cite{BO87} established the surprising much
stronger fact that two algebraic vector bundles on $\mathbb{S}^{2}$
are isomorphic as algebraic vector bundles if and only if their associated
topological vector bundles on $S^{2}$ are isomorphic as topological
vector bundles. Summing up:
\begin{prop}
\label{prop:Alg-equal-Top} The map which associates to a real algebraic
vector bundle $p:F\rightarrow\mathbb{S}^{2}$ on $\mathbb{S}^{2}$
the topological vector bundle $p(\mathbb{R}):\mathbb{F}(\mathbb{R})\rightarrow\mathbb{S}^{2}(\mathbb{R})$
on $\mathbb{S}^{2}(\mathbb{R})\simeq S^{2}$ induces a one-to-one
correspondence between isomorphism classes of algebraic vector bundles
on $\mathbb{S}^{2}$ and isomorphism classes of topological real vector
bundles on $S^{2}$. 
\end{prop}

In the next paragraphs, we review briefly the classification of topological
real vector bundles on $S^{2}$ and give a new construction of corresponding
algebraic models in the framework of complex varieties with real structure. 

\subsection{\label{subsec:TopVB-S2}Recollection on topological real vector bundles
on $S^{2}$ }

Every topological real vector bundle on $S^{2}$ is orientable, and
there exists a bijection 
\[
\theta:[S^{1},\mathrm{GL}_{r}^{+}(\mathbb{R})]\rightarrow\mathrm{Vect}_{r}^{+}(S^{2})
\]
between the set of homotopy classes of continuous map from the circle
$S^{1}$ to the group $\mathrm{GL}_{r}^{+}(\mathbb{R})$ of invertible
matrices of rank $r$ with positive determinant, and the set of isomorphism
classes of oriented topological real vector bundles of rank $r$ on
$S^{2}$. This bijection can be explicitly realized via the so-called
clutching construction. Namely, viewing $S^{2}$ as the union of its
closed lower and upper hemispheres $S_{z\leq0}^{2}$ and $S_{z\geq0}^{2}$
with common boundary $\partial S_{z\leq0}^{2}=\partial S_{z\geq0}^{2}=\left\{ z=0\right\} \simeq S^{1}$,
a continuous map $f:S^{1}\rightarrow\mathrm{GL}_{r}^{+}(\mathbb{R})$
determines a real vector bundle $\pi:E_{f}\rightarrow S^{2}$ of rank
$r$ obtained as the quotient of $S_{z\leq0}^{2}\times\mathbb{R}^{r}\sqcup S_{z\geq0}^{2}\times\mathbb{R}^{r}$
by identifying $(x,v)\in\partial S_{z\leq0}^{2}\times\mathbb{R}^{r}$
with $(x,f(x)\cdot v)\in\partial S_{z\geq0}^{2}\times\mathbb{R}^{r}$.
The isomorphism class of $E_{f}$ depends only on the homotopy class
of $f$, and the bijection $\theta$ is defined by sending a clutching
map $f:S^{1}\rightarrow\mathrm{GL}_{r}^{+}(\mathbb{R})$ to the vector
bundle $E_{f}$ it determines (see e.g. \cite[Proposition 1.11]{Hbook}). 

Noting that $\mathrm{GL}_{r}^{+}(\mathbb{R})$ retracts onto the special
orthogonal group $\mathrm{SO}_{r}$, we get that $\mathrm{Vect}_{1}^{+}(S^{2})$
consists of the trivial line bundle only, and that $\mathrm{Vect}_{2}^{+}(S^{2})$
is isomorphic to $\pi_{1}(\mathrm{SO}_{2})\simeq\mathbb{Z}$. Identifying
$S^{1}$ and $SO_{2}$ with the set of complex numbers $\alpha=x+iy$
of modulus one, a corresponding collection of clutching maps $f_{n}:S^{1}\rightarrow\mathrm{SO}_{2}$
is simply given by $\alpha\mapsto\alpha^{n}$, $n\in\mathbb{Z}$.
Writing $\alpha=\exp(i\theta)$, $\theta\in\mathbb{R}$, these correspond
equivalently to the rotation matrices 
\[
M_{2}(n)=\exp(i\theta)^{n}=\left(\begin{array}{cc}
\cos n\theta & \sin n\theta\\
-\sin n\theta & \cos n\theta
\end{array}\right).
\]
The real vector bundle corresponding to $M_{2}(n)\in\mathrm{SO}_{2}$
coincides with the image of the underlying real vector bundle of the
complex line bundle $\mathcal{O}_{\mathbb{CP}^{1}}(n)$ on $\mathbb{CP}^{1}$
via the usual diffeomorphism $\mathbb{CP}^{1}\rightarrow S^{2}=\mathbb{C}\cup\left\{ \infty\right\} $
mapping $[z_{0}:z_{1}]$ to $z_{0}/z_{1}$. For instance, the tangent
bundle $TS^{2}\rightarrow S^{2}$ coincides with the image of underlying
real vector bundle of $\mathcal{O}_{\mathbb{CP}^{1}}(2)$. Note that
the underlying real vector bundle of $\mathcal{O}_{\mathbb{CP}^{1}}(-n)$
endowed with the orientation inherited from the complex structure
is equal to the underlying real vector bundle of $\mathcal{O}_{\mathbb{CP}^{1}}(n)$
but equipped with the opposite orientation. 

For every $r\geq3$, $\mathrm{Vect}_{r}^{+}(S^{2})$ is isomorphic
to $\pi_{1}(\mathrm{SO}_{r})\simeq\mathbb{Z}/2\mathbb{Z}$, a corresponding
clutching map being given by the matrix 
\[
\mathrm{diag}((\exp(i\theta),I_{r-2})=\left(\begin{array}{ccc}
\cos\theta & \sin\theta & 0\\
-\sin\theta & \cos\theta & 0\\
0 & 0 & I_{r-2}
\end{array}\right)
\]
where $I_{r-2}$ denote the $(r-2)\times(r-2)$ identity matrix. In
other words, for every $r\geq3$, the unique nontrivial real topological
vector bundle of rank $r$ on $S^{2}$ is the direct sum of the rank
$2$ vector bundle $\pi_{f_{1}}:E_{f_{1}}\rightarrow S^{2}$ corresponding
to $M_{2}(1)$ and of the trivial vector bundle of rank $r-2$. It
also follows from this description that $E_{f_{n}}$ is either $1$-stably
trivial if $n$ is even or $1$-stably isomorphic to $E_{f_{1}}$
is $n$ is odd. 

\subsection{\label{subsec:Algebraic-models-VB-S2}Algebraic models as vector
bundles on the projective quadric}

In view of the description of isomorphism classes of topological real
vector bundles on $S^{2}$ recalled in \S \ref{subsec:TopVB-S2},
to show that every real topological vector bundle $\pi:E\rightarrow S^{2}$
admits an algebraic model, it is enough to show that for every $n\geq1$,
the real topological vector bundle $\pi_{f_{n}}:E_{f_{n}}\rightarrow S^{2}$
corresponding to the underlying real vector bundle of the complex
line bundle $\mathcal{O}_{\mathbb{CP}^{1}}(n)$ on $\mathbb{CP}^{1}$
admits such a model. Models for these bundles were constructed by
Moore \cite{Mo71} and Swan \cite{Sw93} in the form of certain projective
modules on the coordinate ring of the affine surface $\mathbb{S}^{2}$.
The construction we give below is in contrast of geometric nature,
providing models of these bundles in the form of restrictions to $\mathbb{S}^{2}$
of natural algebraic vector bundles on the real projective quadric
$\mathbb{Q}^{2}$. \\

The closed embedding $\mathbb{P}_{\mathbb{C}}^{1}\times\mathbb{P}_{\mathbb{C}}^{1}\rightarrow\mathbb{P}_{\mathbb{C}}^{3}$
defined by 
\[
([x_{0}:x_{1}][y_{0}:y_{1}])\mapsto[X:Y:Z:T]=[x_{0}y_{1}+x_{1}y_{0}:i(x_{1}y_{0}-x_{0}y_{1}):x_{0}y_{0}-x_{1}y_{1}:x_{0}y_{0}+x_{1}y_{1}]
\]
induces an isomorphism $\psi:\mathbb{P}_{\mathbb{C}}^{1}\times\mathbb{P}_{\mathbb{C}}^{1}\stackrel{\simeq}{\longrightarrow}\mathbb{Q}_{\mathbb{C}}^{2}$.
The pull-back $\psi^{*}\sigma_{\mathbb{Q}^{2}}$ of the canonical
real structure $\sigma_{\mathbb{Q}^{2}}$ on the complexification
$\mathbb{Q}_{\mathbb{C}}^{2}$ of $\mathbb{Q}^{2}$ is the real structure
$\sigma=s_{\Delta}\circ(\sigma_{\mathbb{P}_{\mathbb{R}}^{1}}\times\sigma_{\mathbb{P}_{\mathbb{R}}^{1}})$
on $\mathbb{P}_{\mathbb{C}}^{1}\times\mathbb{P}_{\mathbb{C}}^{1}$,
where $s_{\Delta}$ is the algebraic involution which exchanges the
two factors and $\sigma_{\mathbb{P}_{\mathbb{R}}^{1}}$ is the canonical
real structure on $\mathbb{P}_{\mathbb{C}}^{1}=(\mathbb{P}_{\mathbb{R}}^{1})_{\mathbb{C}}$.
Applying the construction explained in \S \ref{subsec:Galois-descent-VB}
to the line bundles 
\[
p_{n}:L_{n}=\mathrm{pr}_{1}^{*}\mathcal{O}_{\mathbb{P}_{\mathbb{C}}^{1}}(n)\rightarrow\mathbb{P}_{\mathbb{C}}^{1}\times\mathbb{P}_{\mathbb{C}}^{1},\quad n\geq0
\]
on $\mathbb{P}_{\mathbb{C}}^{1}\times\mathbb{P}_{\mathbb{C}}^{1}$,
we obtain a collection of algebraic vector bundles $p_{n,\mathbb{R}}:L_{n,\mathbb{R}}\rightarrow\mathbb{Q}^{2}$
of rank $2$ on $\mathbb{Q}^{2}$. 
\begin{lem}
\label{lem:algebraic-model-rk2} For every $n\geq0$, the following
hold: 

a) The complexification $p_{n,\mathbb{C}}:(L_{n,\mathbb{R}})_{\mathbb{C}}\rightarrow\mathbb{P}_{\mathbb{C}}^{1}\times\mathbb{P}_{\mathbb{C}}^{1}$
is isomorphic to $\mathrm{pr}_{1}^{*}\mathcal{O}_{\mathbb{P}_{\mathbb{C}}^{1}}(n)\oplus\mathrm{pr}_{2}^{*}\mathcal{O}_{\mathbb{P}_{\mathbb{C}}^{1}}(n)$, 

b) The topological vector bundle $p_{n,\mathbb{R}}(\mathbb{R}):L_{n,\mathbb{R}}(\mathbb{R})\rightarrow\mathbb{Q}^{2}\left(\mathbb{R}\right)=\mathbb{S}(\mathbb{R})$
is isomorphic to $\pi_{f_{n}}:E_{f_{n}}\rightarrow S^{2}$.
\end{lem}

\begin{proof}
By construction, $(L_{n,\mathbb{R}})_{\mathbb{C}}$ is isomorphic
to $L_{n}\oplus\sigma^{*}L_{n}$. Assertion a) then follows from the
the identity 
\[
\sigma^{*}(\mathrm{pr}_{1}^{*}\mathcal{O}_{\mathbb{P}_{\mathbb{C}}^{1}}(n))=(\sigma_{\mathbb{P}_{\mathbb{R}}^{1}}\times\sigma_{\mathbb{P}_{\mathbb{R}}^{1}})^{*}(s_{\Delta}^{*}(\mathrm{pr}_{1}^{*}\mathcal{O}_{\mathbb{P}_{\mathbb{C}}^{1}}(n)))\simeq\mathrm{pr}_{2}^{*}(\sigma_{\mathbb{P}_{\mathbb{R}}^{1}}^{*}(\mathcal{O}_{\mathbb{P}_{\mathbb{C}}^{1}}(n)))
\]
and the fact that $\sigma_{\mathbb{P}_{\mathbb{R}}^{1}}^{*}(\mathcal{O}_{\mathbb{P}_{\mathbb{C}}^{1}}(n))\simeq\mathcal{O}_{\mathbb{P}_{\mathbb{C}}^{1}}(n)$
as line bundles on $\mathbb{P}_{\mathbb{C}}^{1}$. 

The map $\xi=(\mathrm{id}\times\sigma_{\mathbb{P}_{\mathbb{R}}^{2}})\circ\Delta:\mathbb{P}_{\mathbb{C}}^{1}\rightarrow\mathbb{P}_{\mathbb{C}}^{1}\times\mathbb{P}_{\mathbb{C}}^{1}$,
where $\Delta$ denotes the diagonal embedding, induces a diffeomorphism
between $\mathbb{P}_{\mathbb{C}}^{1}(\mathbb{C})=\mathbb{CP}^{1}\simeq S^{2}$
and the real locus of $\mathbb{P}_{\mathbb{C}}^{1}\times\mathbb{P}_{\mathbb{C}}^{1}$
endowed with the real structure $\sigma$. Assertion b) then follows
from Lemma \ref{lem:Top-Down-Bundle} and the identity 
\[
\xi^{*}(L_{n})(\mathbb{C})=((\mathrm{id}\times\sigma_{\mathbb{P}_{\mathbb{R}}^{2}})\circ\Delta)^{*}(\mathrm{pr}_{1}^{*}\mathcal{O}_{\mathbb{P}_{\mathbb{C}}^{1}}(n))(\mathbb{C})\simeq\mathcal{O}_{\mathbb{CP}^{1}}(n)
\]
which holds by construction of $L_{n}$. 
\end{proof}
\begin{rem}
\label{rem:Swan-Dual}Via the isomorphism $\psi:\mathbb{P}_{\mathbb{C}}^{1}\times\mathbb{P}_{\mathbb{C}}^{1}\stackrel{\simeq}{\rightarrow}\mathbb{Q}_{\mathbb{C}}^{2}$,
the line bundle $L_{n}=\mathrm{pr}_{1}^{*}\mathcal{O}_{\mathbb{P}_{\mathbb{C}}^{1}}(n)$
coincides with the line bundle on $\mathbb{Q}_{\mathbb{C}}^{2}$ associated
to the Cartier divisor $nC$, where $C$ is the irreducible and reduced
curve $\{X+iY=T-Z=0\}$ on $\mathbb{Q}_{\mathbb{C}}^{2}$. The $\Gamma(\mathbb{S}^{2},\mathcal{O}_{\mathbb{S}^{2}})$-module
of global sections $\Gamma(\mathbb{S}^{2},L_{n,\mathbb{R}})$ of the
restriction of $L_{n,\mathbb{R}}$ to $\mathbb{S}^{2}=\mathbb{Q}^{2}\setminus\{T=0\}$
then coincides with the invertible $\Gamma(\mathbb{S}_{\mathbb{C}}^{2},\mathcal{O}_{\mathbb{S}_{\mathbb{C}}^{2}})$-module
$\mathfrak{p}^{n}=(x+iy,1-z)^{n}\subset\Gamma(\mathbb{S}_{\mathbb{C}}^{2},\mathcal{O}_{\mathbb{S}_{\mathbb{C}}^{2}})$,
viewed as a projective $\Gamma(\mathbb{S}^{2},\mathcal{O}_{\mathbb{S}^{2}})$-module
of rank $2$ via the inclusion $\Gamma(\mathbb{S}^{2},\mathcal{O}_{\mathbb{S}^{2}})\hookrightarrow\Gamma(\mathbb{S}_{\mathbb{C}}^{2},\mathcal{O}_{\mathbb{S}_{\mathbb{C}}^{2}})$.
We thus recover geometrically the construction given by Swan in \cite{Sw93}. 

Note also that since the inverse image of $\mathbb{Q}_{\mathbb{C}}^{2}\setminus\mathbb{S}_{\mathbb{C}}^{2}$
by $\psi$ is the irreducible curve $\Gamma=\{x_{0}y_{0}+x_{1}y_{1}=0\}$
of type $(1,1)$ in the divisor class group of $\mathbb{P}_{\mathbb{C}}^{1}\times\mathbb{P}_{\mathbb{C}}^{1}$,
the restriction of $(L_{n,\mathbb{R}})_{\mathbb{C}}\simeq\mathrm{pr}_{1}^{*}\mathcal{O}_{\mathbb{P}_{\mathbb{C}}^{1}}(n)\oplus\mathrm{pr}_{2}^{*}\mathcal{O}_{\mathbb{P}_{\mathbb{C}}^{1}}(n)$
to $\mathbb{P}_{\mathbb{C}}^{1}\times\mathbb{P}_{\mathbb{C}}^{1}\setminus\Gamma\simeq\mathbb{S}_{\mathbb{C}}^{2}$
is isomorphic to that of $\mathrm{pr}_{1}^{*}\mathcal{O}_{\mathbb{P}_{\mathbb{C}}^{1}}(n)\oplus\mathrm{pr}_{1}^{*}\mathcal{O}_{\mathbb{P}_{\mathbb{C}}^{1}}(-n)\simeq L_{n}\oplus L_{n}^{\vee}$,
where $L_{n}^{\vee}$ denotes the dual of $L_{n}$. 
\end{rem}

\section{Real forms of the trivial bundle $\mathbb{S}^{2}\times\mathbb{A}_{\mathbb{R}}^{2}$ }
\begin{notation}
\label{nota:Vn}For every $n\geq0$, we let $q_{n}:V_{n}\rightarrow\mathbb{S}^{2}$
be the restriction to $\mathbb{S}^{2}\subset\mathbb{Q}^{2}$ of the
rank $2$ vector bundle $p_{n,\mathbb{R}}:L_{n,\mathbb{R}}\rightarrow\mathbb{Q}^{2}$
constructed in \S \ref{subsec:Algebraic-models-VB-S2}.
\end{notation}

\subsection{Proof of Theorem \ref{thm:MainTh-1}}

The following proposition implies Theorem \ref{thm:MainTh-1}: 
\begin{prop}
\label{prop:Real-forms-trivial-bundle}The real algebraic varieties
$V_{n}$, $n\geq0$, are pairwise non isomorphic real forms of $V_{0}=\mathbb{S}^{2}\times\mathbb{A}_{\mathbb{R}}^{2}$.
\end{prop}

The proof is a combination of Lemma \ref{lem:trivial-complex} and
Lemma \ref{lem:non-iso-total-spaces} below which show that the real
algebraic varieties $V_{n}$, $n\geq0$, are pairwise non isomorphic
with isomorphic complexifications $V_{n,\mathbb{C}}\simeq\mathbb{S}_{\mathbb{C}}^{2}\times\mathbb{A}_{\mathbb{C}}^{2}$.
\\

By construction, the rank $2$ vector bundles $p_{n,\mathbb{R}}:L_{n,\mathbb{R}}\rightarrow\mathbb{Q}^{2}$,
$n\geq0$, on $\mathbb{Q}^{2}$ have pairwise non-isomorphic complexifications
$(L_{n,\mathbb{R}})_{\mathbb{C}}\simeq\mathrm{pr}_{1}^{*}\mathcal{O}_{\mathbb{P}_{\mathbb{C}}^{1}}(n)\oplus\mathrm{pr}_{2}^{*}\mathcal{O}_{\mathbb{P}_{\mathbb{C}}^{1}}(n)\rightarrow\mathbb{P}_{\mathbb{C}}^{1}\times\mathbb{P}_{\mathbb{C}}^{1}$.
The next lemma shows in contrast that their restrictions to $\mathbb{S}^{2}\subset\mathbb{Q}^{2}$
all have isomorphic complexifications: 
\begin{lem}
\label{lem:trivial-complex}For every $n\geq0$, the complexification
$q_{n,\mathbb{C}}:V_{n,\mathbb{C}}\rightarrow\mathbb{S}_{\mathbb{C}}^{2}$
of $q_{n}:V_{n}\rightarrow\mathbb{S}^{2}$ is isomorphic to the trivial
vector bundle $\mathrm{pr}_{1}:\mathbb{S}_{\mathbb{C}}^{2}\times\mathbb{A}_{\mathbb{C}}^{2}\rightarrow\mathbb{S}_{\mathbb{C}}^{2}$. 
\end{lem}

\begin{proof}
Since the Picard group of $\mathbb{S}^{2}$ is equal to its divisor
class group which is trivial, the determinant $\det(V_{n})$ of $V_{n}$
is isomorphic to the trivial line bundle on $\mathbb{S}^{2}$. This
implies in turn that $\det(V_{n,\mathbb{C}})$ is the trivial line
bundle on $\mathbb{S}_{\mathbb{C}}^{2}$, a fact which also follows
more concretely from the observation made in Remark \ref{rem:Swan-Dual}
that $V_{n,\mathbb{C}}$ is isomorphic to the direct sum of a line
bundle and its dual. Since by a general result of Murthy \cite{Mu69},
every algebraic vector bundle of rank $2$ on $\mathbb{S}_{\mathbb{C}}^{2}$
splits a trivial factor, hence is isomorphic to the direct sum of
its determinant and a trivial line bundle, we conclude that for every
$n\geq0$, $V_{n,\mathbb{C}}$ is isomorphic to the trivial vector
bundle of rank $2$ on $\mathbb{S}_{\mathbb{C}}^{2}$ (see \S \ref{subsec:Explicit-RealStruct}
below for the construction of explicit isomorphisms $V_{n,\mathbb{C}}\simeq\mathbb{S}_{\mathbb{C}}^{2}\times\mathbb{A}_{\mathbb{C}}^{2}$). 
\end{proof}
By $\S$ \ref{subsec:TopVB-S2} and Lemma \ref{lem:algebraic-model-rk2},
the algebraic vectors bundles $q_{n}:V_{n}\rightarrow\mathbb{S}^{2}$,
$n\geq0$, are pairwise non-isomorphic as vector bundles over $\mathbb{S}^{2}$.
The following result then implies the stronger fact that their total
spaces are pairwise non isomorphic as abstract algebraic varieties: 
\begin{lem}
\label{lem:non-iso-total-spaces} The total spaces of two algebraic
vector bundles $q:V\rightarrow\mathbb{S}^{2}$ and $q':V'\rightarrow\mathbb{S}^{2}$
are isomorphic as abstract real algebraic varieties if and only if
$q:V\rightarrow\mathbb{S}^{2}$ and $q':V'\rightarrow\mathbb{S}^{2}$
are isomorphic as vector bundles. 
\end{lem}

\begin{proof}
Let $\Psi:V\rightarrow V'$ be an isomorphism of abstract real algebraic
varieties. First note that every morphism $f:\mathbb{A}_{\mathbb{R}}^{1}\rightarrow\mathbb{S}^{2}$
is constant. Indeed, otherwise, since $\mathbb{S}^{2}$ is affine
hence does not contain complete curves, $f$ would extend to a nonconstant
morphism $\overline{f}:\mathbb{P}_{\mathbb{R}}^{1}\rightarrow\mathbb{Q}^{2}$
mapping the real point $\mathbb{P}_{\mathbb{R}}^{1}\setminus\mathbb{A}_{\mathbb{R}}^{1}$
to a point of $\mathbb{Q}^{2}\setminus\mathbb{S}^{2}$. But this is
impossible since the latter is a conic without real point. The restriction
of $q'\circ\Psi:V\rightarrow V'$ to every fiber of $q$ over a real
point of $\mathbb{S}^{2}$ is thus constant. Since the set of points
$s$ of $\mathbb{S}^{2}$ such that $\dim((q'\circ\Psi)(q^{-1}(s)))=0$
is closed in $\mathbb{S}^{2}$ and $\mathbb{S}^{2}(\mathbb{R})$ is
Zariski dense in $\mathbb{S}^{2}$, it follows that $q'\circ\Psi$
is constant on the fibers of $q$, hence descends to a unique automorphism
$\psi$ of $\mathbb{S}^{2}$ such that $q'\circ\Psi=\psi\circ q$.
This implies in turn that $\Psi$ induces an isomorphism $\tilde{\Psi}:V\rightarrow\tilde{V}=\psi^{*}V'$
of schemes over $\mathbb{S}^{2}$. Now it follows from \cite[Lemma 1.3]{EaH73}
that $p:V\rightarrow\mathbb{S}^{2}$ and $\tilde{p}=p'\circ\tilde{\Psi}:\tilde{V}\rightarrow\mathbb{S}^{2}$
are isomorphic as algebraic vector bundles over $\mathbb{S}^{2}$.
Let us briefly recall the argument for the sake of completeness: since
$V$ and $\tilde{V}$ are vector bundles, their relative tangent bundles
$T_{V/\mathbb{S}^{2}}$ and $T_{\tilde{V}/\mathbb{S}^{2}}$ are isomorphic
to $p^{*}V$ and $\tilde{p}^{*}\tilde{V}$ respectively. Letting $\alpha:\mathbb{S}^{2}\rightarrow V$
be any section of $p$, the composition $\tilde{\alpha}=\tilde{\Psi}\circ\alpha$
is a section of $\tilde{p}$, and the relative differential $d\tilde{\Psi}_{/\mathbb{S}^{2}}:T_{V/\mathbb{S}^{2}}\rightarrow\tilde{\Psi}^{*}T_{\tilde{V}/\mathbb{S}^{2}}$
of $\tilde{\Psi}$ over $\mathbb{S}^{2}$ then induces an isomorphism
\[
\alpha^{*}d\tilde{\Psi}_{/\mathbb{S}^{2}}:V\simeq\alpha^{*}T_{V/\mathbb{S}^{2}}\stackrel{\simeq}{\longrightarrow}\alpha^{*}\tilde{\Psi}^{*}T_{\tilde{V}/\mathbb{S}^{2}}=\tilde{\alpha}^{*}T_{\tilde{V}/\mathbb{S}^{2}}\simeq\tilde{V}
\]
of algebraic vector bundles over $\mathbb{S}^{2}$. To complete the
proof, it thus remains to show that $\psi^{*}V'$ is isomorphic to
$V'$ as algebraic vector bundles over $\mathbb{S}^{2}$. By virtue
of Proposition \ref{prop:Alg-equal-Top}, it suffices to show that
the pull-back of $V'(\mathbb{R})$ by the induced diffeomorphism $\psi(\mathbb{R})$
of $\mathbb{S}^{2}(\mathbb{R})\simeq S^{2}$ is isomorphic to $V'(\mathbb{R})$
as a topological real vector bundle. Since the mapping class group
of $S^{2}$ is isomorphic to $\mathbb{Z}/2\mathbb{Z}$, $\psi(\mathbb{R})^{*}V'(\mathbb{R})$
is either isomorphic to $V'(\mathbb{R})$ if $\psi(\mathbb{R})$ is
orientation preserving, or to $V'(\mathbb{R})$ but endowed with the
opposite orientation otherwise. So in each case $\psi(\mathbb{R})^{*}V'(\mathbb{R})\simeq V'(\mathbb{R})$
and the assertion follows. 
\end{proof}
\begin{rem}
In the special case of the rank $2$ vector bundles $q_{n}:V_{n}\rightarrow\mathbb{S}^{2}$,
$n\geq0$, it is well-known that the associated real topological fourfolds
$V_{n}(\mathbb{R})\simeq\mathcal{O}_{\mathbb{CP}^{1}}(n)$ are actually
even pairwise non-homeomorphic. This can be seen by comparing their
respective first homology groups at infinity $H_{1}^{\infty}(\mathcal{O}_{\mathbb{CP}^{1}}(n);\mathbb{Z})$,
defined as the limit over exhaustions of $\mathcal{O}_{\mathbb{CP}^{1}}(n)$
by compact subsets $K_{i}$ of the homology groups $H_{1}(\mathcal{O}_{\mathbb{CP}^{1}}(n)\setminus K_{i};\mathbb{Z})$.
Since $\mathcal{O}_{\mathbb{CP}^{1}}(n)$ is homeomorphic to the complement
in the Hirzebruch surface $\beta_{n}:\mathbb{F}_{n}=\mathbb{P}(\mathcal{O}_{\mathbb{CP}^{1}}(-n)\oplus\mathcal{O}_{\mathbb{CP}^{1}})\rightarrow\mathbb{CP}^{1}$
of a section $H_{0,n}$ of $\beta_{n}$ with self-intersection $-n$,
which is unique if $n\neq0$, it follows by excision that $H_{1}^{\infty}(\mathcal{O}_{\mathbb{CP}^{1}}(n);\mathbb{Z})$
is isomorphic to the first homology group of a pointed tubular neighborhood
$T_{*}(H_{0,n})$ in $\mathbb{F}_{n}$, i.e. a tubular neighborhood
of $H_{0;n}$ in $\mathbb{F}_{n}$ with $H_{0;n}$ removed from it.
We conclude that 
\[
H_{1}^{\infty}(\mathcal{O}_{\mathbb{CP}^{1}}(n);\mathbb{Z})\simeq H_{1}(T_{*}(H_{0,n});\mathbb{Z})\simeq\mathbb{Z}/\deg\mathcal{N}_{H_{0,n}/\mathbb{F}_{n}}\mathbb{Z}\simeq\mathbb{Z}/n\mathbb{Z}
\]
where $\mathcal{N}_{H_{0,n}/\mathbb{F}_{n}}\simeq\mathcal{O}_{\mathbb{CP}^{1}}(-n)$
denotes the normal bundle of $H_{0,n}$ in $\mathbb{F}_{n}$. 
\end{rem}

For every $n\geq0$, the real algebraic variety $V_{n}\times\mathbb{A}_{\mathbb{R}}^{1}$
is the total space of an algebraic vector bundle $q_{n}\circ\mathrm{pr_{1}:V_{n}\times\mathbb{A}_{\mathbb{R}}^{1}\rightarrow\mathbb{S}^{2}}$
of rank $3$ on $\mathbb{S}^{2}$. By combining the classification
of topological real vector bundles on $S^{2}$ given in \S \ref{subsec:TopVB-S2}
with Proposition \ref{prop:Alg-equal-Top} and Lemma \ref{lem:non-iso-total-spaces},
we obtain the following generalization of Hochster's counter-example
to the Zariski Cancellation Problem \cite{Ho72} which, in our notation,
corresponds to the case of the vector bundle $p_{2}:V_{2}\rightarrow\mathbb{S}^{2}$,
isomorphic to the tangent bundle $T\mathbb{S}^{2}\rightarrow\mathbb{S}^{2}$
of $\mathbb{S}^{2}$. 
\begin{cor}
\label{cor:Stable-Iso}The real algebraic variety $V_{n}\times\mathbb{A}_{\mathbb{R}}^{1}$
is isomorphic to $V_{0}\times\mathbb{A}_{\mathbb{R}}^{1}$ if $n$
is even or to $V_{1}\times\mathbb{A}_{\mathbb{R}}^{1}$ if $n$ is
odd. As a consequence, the real algebraic varieties $V_{2p}$ $($resp.
$V_{2p+1}$$)$, $p\geq0$, form a family of pairwise non isomorphic
rational factorial real algebraic varieties with isomorphic cylinders
$V_{2p}\times\mathbb{A}_{\mathbb{R}}^{1}$ $($resp. $V_{2p+1}\times\mathbb{A}_{\mathbb{R}}^{1}$$)$. 
\end{cor}

\section{Examples and applications }

\subsection{\label{subsec:Explicit-RealStruct}Explicit family of non-equivalent
real structures on $\mathbb{S}_{\mathbb{C}}^{2}\times\mathbb{A}_{\mathbb{C}}^{2}$}

By Lemma \ref{lem:trivial-complex}, the complexifications $q_{n,\mathbb{C}}:V_{n,\mathbb{C}}\rightarrow\mathbb{S}_{\mathbb{C}}^{2}$
of the rank $2$ vector bundles $q_{n}:V_{n}\rightarrow\mathbb{S}^{2}$,
$n\geq0$, are all isomorphic to the trivial vector bundle $\mathrm{pr}_{1}:\mathbb{S}_{\mathbb{C}}^{2}\times\mathbb{A}_{\mathbb{C}}^{2}\rightarrow\mathbb{S}_{\mathbb{C}}^{2}$.
In fact, we have the following more explicit description:
\begin{prop}
\label{prop:real-Struct}For $n\geq1$, let $P_{n},Q_{n}\in\mathbb{R}[z]\subset\Gamma(\mathbb{S}_{\mathbb{C}}^{2},\mathcal{O}_{\mathbb{S}_{\mathbb{C}}^{2}})$
be any polynomials such that 
\[
(1+z)^{n}P_{n}(z)+(1-z)^{n}Q_{n}(z)=1.
\]
Then the following hold:

a) The composition of the canonical product real structure $\Sigma_{0}=\sigma_{\mathbb{S}^{2}}\times\sigma_{\mathbb{A}_{\mathbb{R}}^{2}}$
on $\mathbb{S}_{\mathbb{C}}^{2}\times\mathbb{A}_{\mathbb{C}}^{2}$
with the automorphism of the trivial bundle $\mathbb{S}_{\mathbb{C}}^{2}\times\mathbb{A}_{\mathbb{C}}^{2}$
defined by the matrix 
\[
A_{n}=\left(\begin{array}{cc}
(x-iy)^{n}(P_{n}+Q_{n}) & (1-z)^{n}-(1+z)^{n}\\
-(1+z)^{n}P_{n}^{2}+(1-z)^{n}Q_{n}^{2} & -(x+iy)^{n}(P_{n}+Q_{n})
\end{array}\right)\in\mathrm{GL}_{2}(\Gamma(\mathbb{S}_{\mathbb{C}}^{2},\mathcal{O}_{\mathbb{S}_{\mathbb{C}}^{2}}))
\]
defines a real structure $\Sigma_{n}$ on $\mathbb{S}_{\mathbb{C}}^{2}\times\mathbb{A}_{\mathbb{C}}^{2}$. 

b) There exists an isomorphism $\Theta_{n}:V_{n,\mathbb{C}}\stackrel{\simeq}{\longrightarrow}\mathbb{S}_{\mathbb{C}}^{2}\times\mathbb{A}_{\mathbb{C}}^{2}$
of vector bundles over $\mathbb{S}_{\mathbb{C}}^{2}$ such that $\Sigma_{n}\circ\Theta_{n}=\Theta_{n}\circ\sigma_{V_{n}}$,
where $\sigma_{V_{n}}$ denotes the canonical real structure on $V_{n,\mathbb{C}}$. 
\end{prop}

\begin{proof}
The fact that $A_{n}$ defines an automorphism $j_{n}$ of the trivial
bundle $\mathbb{S}_{\mathbb{C}}^{2}\times\mathbb{A}_{\mathbb{C}}^{2}$
such that $(j_{n}\circ\Sigma_{0})^{2}=\mathrm{id}_{\mathbb{S}_{\mathbb{C}}^{2}\times\mathbb{A}_{\mathbb{C}}^{2}}$
follows from a direct calculation. To construct the isomorphism $\Theta_{n}$,
we recall from Remark \ref{rem:Swan-Dual} that the vector bundle
$q_{n,\mathbb{C}}:V_{n,\mathbb{C}}\rightarrow\mathbb{S}_{\mathbb{C}}^{2}$
is isomorphic to the direct sum of the line bundle $E_{n}\rightarrow\mathbb{S}_{\mathbb{C}}^{2}$
associated to the locally free sheaf $\mathcal{O}_{\mathbb{S}_{\mathbb{C}}^{2}}(nC)$,
where $C=\left\{ x+iy=1-z=0\right\} $ and of the line bundle $\sigma_{\mathbb{S}^{2}}^{*}E_{n}$
associated to the locally free sheaf $\mathcal{O}_{\mathbb{S}_{\mathbb{C}}^{2}}(n\sigma_{\mathbb{S}^{2}}^{-1}(C))$,
where $\sigma_{\mathbb{S}^{2}}^{-1}(C)=\left\{ x-iy=1-z=0\right\} $.
The canonical real structure $\sigma_{V_{n}}$ on $V_{n,\mathbb{C}}$
coincides via this isomorphism with the natural lift of $\sigma_{\mathbb{S}^{2}}$
to a real structure $\tilde{\sigma}_{n}$ on $E_{n}\oplus\sigma_{\mathbb{S}^{2}}^{*}E_{n}$
defined in \S \ref{subsec:Galois-descent-VB}. 

The surface $\mathbb{S}_{\mathbb{C}}^{2}$ is covered by the $\sigma_{\mathbb{S}^{2}}$-invariant
principal affine open subsets $U_{\pm}=\mathbb{S}_{\mathbb{C}}^{2}\setminus\left\{ 1\pm z\neq0\right\} $.
By definition of $C$ and $\sigma_{\mathbb{S}^{2}}^{-1}(C)$, we have
$C\cap U_{-}=\sigma_{\mathbb{S}^{2}}^{-1}(C)\cap U_{-}=\emptyset$,
whereas $C\cap U_{+}$ and $\sigma_{\mathbb{S}^{2}}^{-1}(C)\cap U_{+}$
are principal divisors due to the relation $\left(1-z\right)=(1+z)^{-1}(x-iy)(x+iy)$
which holds in the coordinate ring of $U_{+}$. The choice of local
equations $\{1,x+iy\}$ and $\{1-z,(1+z)^{-1}(x-iy)\}$ for $C$ and
$\sigma_{\mathbb{S}^{2}}^{-1}(C)$ on $U_{-}$ and $U_{+}$ induces
local trivializations 
\[
\begin{array}{ccc}
\gamma_{n,\pm}:E_{n}\oplus\sigma_{\mathbb{S}^{2}}^{*}E_{n}|_{U_{\pm1}} & \stackrel{\simeq}{\longrightarrow} & U_{\pm}\times\mathrm{Spec}(\mathbb{C}[t_{n,\pm},t_{n,\pm}'])\end{array}
\]
for which the isomorphism $\psi_{n}=\gamma_{n,+}\circ\gamma_{n,-}^{-1}|_{U_{+}\cap U_{-}}$
is given by the matrix 
\[
D_{n}=\left(\begin{array}{cc}
\left(x+iy\right)^{n} & 0\\
0 & \left(x+iy\right)^{-n}
\end{array}\right)\in\mathrm{SL}_{2}(\Gamma(U_{+}\cap U_{-},\mathcal{O}_{\mathbb{S}_{\mathbb{C}}^{2}})).
\]
A direct calculation using the relation $(1+z)^{n}P_{n}(z)+(1-z)^{n}Q_{n}(z)=1$
then shows that $D_{n}=M_{n,+}^{-1}M_{n,-}$, where 
\[
M_{n,+}=\left(\begin{array}{cc}
\left(\frac{x-iy}{1+z}\right)^{n} & (1+z)^{n}\\
-P_{n} & (x+iy)^{n}Q_{n}
\end{array}\right)\quad\textrm{and}\quad M_{n,-}=\left(\begin{array}{cc}
(1-z)^{n} & \left(\frac{x-iy}{1-z}\right)^{n}\\
-(x+iy)^{n}P_{n} & Q_{n}
\end{array}\right)
\]
are elements of $\mathrm{SL}_{2}(\Gamma(U_{+},\mathcal{O}_{\mathbb{S}_{\mathbb{C}}^{2}}))$
and $\mathrm{SL}_{2}(\Gamma(U_{-},\mathcal{O}_{\mathbb{S}_{\mathbb{C}}^{2}}))$
respectively. It follows that the local trivilizations 
\[
M_{n,\pm}\circ\gamma_{n,\pm}:E_{n}\oplus\sigma_{\mathbb{S}^{2}}^{*}E_{n}|_{U_{\pm}}\stackrel{\simeq}{\longrightarrow}U_{\pm}\times\mathbb{A}_{\mathbb{C}}^{2}
\]
glue to global one $\Theta_{n}:E_{n}\oplus\sigma_{\mathbb{S}^{2}}^{*}E_{n}\stackrel{\simeq}{\rightarrow}\mathbb{S}_{\mathbb{C}}^{2}\times\mathbb{A}_{\mathbb{C}}^{2}$. 

With our choice of local generators, the images of the restrictions
of the real structure $\tilde{\sigma}_{n}$ under the local trivializations
$\gamma_{n,\pm}$ are given locally on the open cover $\{U_{\pm}\}$
by the composition of $\sigma_{\mathbb{S}^{2}}\times\sigma_{\mathbb{A}_{\mathbb{R}}^{2}}|_{U_{\pm1}\times\mathbb{A}_{\mathbb{C}}^{2}}$
with the involutions of the trivial bundles $U_{\pm}\times\mathbb{A}_{\mathbb{C}}^{2}$
with respective matrices 
\[
J_{n,\pm}=\left(\begin{array}{cc}
0 & (1\pm z)^{-n}\\
\left(1\pm z\right)^{n} & 0
\end{array}\right).
\]
A direct computation then confirms that the local real structures
given by the compositions 
\[
M_{n,\pm}\circ J_{n,\pm}(\sigma_{\mathbb{S}^{2}}\times\sigma_{\mathbb{A}_{\mathbb{R}}^{2}}|_{U_{\pm1}})\circ M_{n,\pm}^{-1}|_{U_{\pm}\times\mathbb{A}_{\mathbb{C}}^{2}}
\]
 glue to a global one on $\mathbb{S}_{\mathbb{C}}^{2}\times\mathbb{A}_{\mathbb{C}}^{2}$
equal to $\Sigma_{n}$, for which we have by construction $\Theta_{n}\circ\tilde{\sigma}_{n}=\Sigma_{n}\circ\Theta_{n}$. 
\end{proof}
\begin{example}
\label{exa:Cases-1-2}For $n=1$ and $2$, one can choose for instance
$P_{1}=Q_{1}=1/2$, $P_{2}=(2-z)/4$ and $Q_{2}=(2+z)/4$ to obtain
respectively 
\[
A_{1}=\left(\begin{array}{cc}
x-iy & -2z\\
-\frac{1}{2}z & -x-iy
\end{array}\right)\quad\textrm{and}\quad A_{2}=\left(\begin{array}{cc}
\left(x-iy\right)^{2} & -4z\\
\frac{1}{4}z(z^{2}-2) & -(x+iy)^{2}
\end{array}\right).
\]
\end{example}

\begin{cor}
With the notation of Proposition \ref{prop:real-Struct}, the following
hold:

a) The real structures $\Sigma_{n}$, $n\geq0$ on $\mathbb{S}_{\mathbb{C}}^{2}\times\mathbb{A}_{\mathbb{C}}^{2}$
are pairwise non-equivalent, 

b) The real structure $\Sigma_{n}\times\sigma_{\mathbb{A}_{\mathbb{R}}^{1}}$
on $\mathbb{S}_{\mathbb{C}}^{2}\times\mathbb{A}_{\mathbb{C}}^{3}$
is equivalent to $\Sigma_{0}\times\sigma_{\mathbb{A}_{\mathbb{R}}^{1}}=\sigma_{\mathbb{S}^{2}}\times\sigma_{\mathbb{A}_{\mathbb{R}}^{3}}$
 if $n$ is even and to $\Sigma_{1}\times\sigma_{\mathbb{A}_{\mathbb{R}}^{1}}$
if $n$ is odd.
\end{cor}

\begin{proof}
The first assertion follows from Proposition \ref{prop:Real-forms-trivial-bundle}
and Proposition \ref{prop:real-Struct} since by construction $\mathbb{S}_{\mathbb{C}}^{2}\times\mathbb{A}_{\mathbb{C}}^{2}$
endowed with the real structure $\Sigma_{n}$ corresponds to the algebraic
vector bundle $q_{n}:V_{n}\rightarrow\mathbb{S}^{2}$. Since the variety
$(\mathbb{S}_{\mathbb{C}}^{2}\times\mathbb{A}_{\mathbb{C}}^{2})\times\mathbb{A}_{\mathbb{C}}^{1}$
endowed with the real structure $\Sigma_{n}\times\sigma_{\mathbb{A}_{\mathbb{R}}^{1}}$
corresponds in turn to the algebraic vector bundle $V_{n}\times\mathbb{A}_{\mathbb{R}}^{1}$
on $\mathbb{S}^{2}$, the second assertion follows from Corollary
\ref{cor:Stable-Iso}. 
\end{proof}
\begin{example}
Corresponding to the classical fact that the tangent bundle $T\mathbb{S}^{2}\rightarrow\mathbb{S}^{2}$
is $1$-stably trivial, the real structure $\Sigma_{2}\times\sigma_{\mathbb{A}_{\mathbb{R}}^{1}}$,
defined as the composition of the automorphism $\hat{j}_{2}$ of $\mathbb{S}_{\mathbb{C}}^{2}\times\mathbb{A}_{\mathbb{C}}^{3}$
defined by the matrix 
\[
\hat{A}_{2}=\left(\begin{array}{cc}
A_{2} & 0\\
0 & 1
\end{array}\right)\in\mathrm{GL}_{3}(\Gamma(\mathbb{S}_{\mathbb{C}}^{2},\mathcal{O}_{\mathbb{S}_{\mathbb{C}}^{2}})),
\]
where $A_{2}$ is the matrix in Example \ref{exa:Cases-1-2} with
the canonical real structure $\Sigma_{0}\times\sigma_{\mathbb{A}_{\mathbb{R}}^{1}}$,
is equivalent to $\Sigma_{0}\times\sigma_{\mathbb{A}_{\mathbb{R}}^{1}}$.
By definition, this amounts to the identity $\psi\circ(\Sigma_{2}\times\sigma_{\mathbb{A}_{\mathbb{R}}^{1}})=(\Sigma_{0}\times\sigma_{\mathbb{A}_{\mathbb{R}}^{1}})\circ\psi$
for some automorphism $\psi$ of the trivial bundle $\mathbb{S}_{\mathbb{C}}^{2}\times\mathbb{A}_{\mathbb{C}}^{3}$.
Rewriting this identity in the form 
\[
(\hat{j}_{2}\times\mathrm{id})=\psi^{-1}\circ(\Sigma_{0}\times\sigma_{\mathbb{A}_{\mathbb{R}}^{1}})\circ\psi\circ(\Sigma_{0}\times\sigma_{\mathbb{A}_{\mathbb{R}}^{1}})^{-1},
\]
we see that it holds for instance for the automorphism $\psi$ defined
by the following matrix 
\[
C=\left(\begin{array}{ccc}
\frac{1}{2}y(x+iy)+\frac{i}{4}z^{2} & iz & x\\
-\frac{1}{2}x(x+iy)-\frac{1}{4}z^{2} & z & y\\
\frac{1}{4}z(y-ix) & -(y+ix) & z
\end{array}\right)\in\mathrm{GL}_{3}(\Gamma(\mathbb{S}_{\mathbb{C}}^{2},\mathcal{O}_{\mathbb{S}_{\mathbb{C}}^{2}})).
\]
Indeed, a direct computation shows that $(\Sigma_{0}\times\sigma_{\mathbb{A}_{\mathbb{R}}^{1}})\circ\psi\circ(\Sigma_{0}\times\sigma_{\mathbb{A}_{\mathbb{R}}^{1}})^{-1}$
is defined by the matrix 
\[
\overline{C}=\left(\begin{array}{ccc}
\frac{1}{2}y(x-iy)-\frac{i}{4}z^{2} & -iz & x\\
-\frac{1}{2}x(x-iy)-\frac{1}{4}z^{2} & z & y\\
\frac{1}{4}z(y+ix) & -(y-ix) & z
\end{array}\right)\in\mathrm{GL}_{3}(\Gamma(\mathbb{S}_{\mathbb{C}}^{2},\mathcal{O}_{\mathbb{S}_{\mathbb{C}}^{2}}))
\]
and that $\hat{A}_{2}=C^{-1}\cdot\overline{C}$. 
\end{example}

\subsection{Proof of Theorem \ref{thm:MainTh-2}}
\begin{lem}
\label{lem:Rational-product}For every $n\geq1$, there exists a smooth
rational real affine variety $X$ of dimension $n$ such that $X(\mathbb{R})\neq\emptyset$
and such that $X_{\mathbb{C}}$ is a of log-general type, with trivial
automorphism group $\mathrm{Aut}(X_{\mathbb{C}})$. 
\end{lem}

\begin{proof}
Indeed, it suffices to take for $X$ the complement in $\mathbb{P}_{\mathbb{R}}^{n}$
of smooth real hypersurface of degree $d>n+1$ (non-connected in the
case $n=1$). 
\end{proof}
Theorem \ref{thm:MainTh-2} is then a consequence of the following
more precise result: 
\begin{prop}
Let $X$ be a smooth rational real affine variety $X$ as in Lemma
\ref{lem:Rational-product}. Then the rational real affine variety
$\mathbb{S}^{2}\times\mathbb{A}_{\mathbb{R}}^{2}\times X$ has at
least countably infinitely many pairwise distinct real forms.
\end{prop}

\begin{proof}
With the notation \ref{nota:Vn}, it suffices to check that the varieties
$V_{n}\times X$, $n\geq0$, are pairwise non isomorphic since their
complexifications are all isomorphic to $\mathbb{S}_{\mathbb{C}}^{2}\times\mathbb{A}_{\mathbb{C}}^{2}\times X_{\mathbb{C}}$
by Lemma \ref{lem:trivial-complex}. So suppose given an isomorphism
of abstract real algebraic varieties $h:V_{n}\times X\rightarrow V_{m}\times X$
for some $n,m\geq0$. The complexification $h_{\mathbb{C}}$ of $h$
is then an automorphism of $\mathbb{S}_{\mathbb{C}}^{2}\times\mathbb{A}_{\mathbb{C}}^{2}\times X_{\mathbb{C}}$.
Since $\mathbb{S}_{\mathbb{C}}^{2}\times\mathbb{A}_{\mathbb{C}}^{2}$
is $\mathbb{A}^{1}$-ruled whereas $X_{\mathbb{C}}$ is of log-general
type by hypothesis, it follows from Iitaka-Fujita strong cancellation
theorem \cite{IiFu77} that there exists a unique automorphism $\xi$
of $X_{\mathbb{C}}$ such that $\mathrm{pr}_{X_{\mathbb{C}}}\circ h_{\mathbb{C}}=\xi\circ\mathrm{pr}_{X_{\mathbb{C}}}$.
By hypothesis, $\xi=\mathrm{id}_{X_{\mathbb{C}}}$ and so, we conclude
that $h$ is actually an isomorphism of schemes over $X$. Since $X(\mathbb{R})\neq\emptyset$,
the restriction of $h$ over any real point $x$ of $X$ is then an
isomorphism of real algebraic varieties $V_{n}\stackrel{\sim}{\rightarrow}V_{m}$,
which implies that $m=n$ by Lemma \ref{lem:non-iso-total-spaces}. 
\end{proof}
\bibliographystyle{amsplain}

\end{document}